\documentclass[a4paper,oneside,10pt]{article}%
\usepackage{amsmath}
\usepackage{amsfonts}
\usepackage{amssymb}
\usepackage{graphicx}
\usepackage{enumerate}
\usepackage{color}
\usepackage[square,numbers,sort&compress]{natbib}%
\usepackage{hyperref}
\usepackage[amsmath,thmmarks]{ntheorem}
\hypersetup{hypertex=true,bookmarks=true,bookmarksnumbered	=true,linktocpage=false,colorlinks=true,linkcolor=blue,anchorcolor=blue,citecolor=red}
\usepackage{indentfirst} 
\usepackage[mathscr]{euscript}
\providecommand{\U}[1]{\protect \rule{.1in}{.1in}}

\newtheorem{theorem}{Theorem}[section]

\newtheorem{corollary}[theorem]{Corollary}

\newtheorem{definition}[theorem]{Definition}

\newtheorem{lemma}[theorem]{Lemma}

\newtheorem{proposition}[theorem]{Proposition}
\newtheorem{remark}[theorem]{Remark}

\newenvironment{proof}[1][Proof]{\noindent \textbf{#1.} }{\  $\Box$}
\numberwithin{equation}{section}

\begin{document}

\makeatletter
\newcommand{\rmnum}[1]{\romannumeral #1}
\newcommand{\Rmnum}[1]{\expandafter\@slowromancap\romannumeral #1@}
\makeatother
{\theoremstyle{nonumberplain}
}

\title{Squared Bessel processes under nonlinear expectation}
\author{Mingshang Hu\textsuperscript{1} \and Renxing Li \textsuperscript{2} \and Xue Zhang\textsuperscript{3}}

\author{Mingshang Hu\textsuperscript{1}   
	\and Renxing Li\textsuperscript{2,}\thanks{Corresponding author. E-mail address: 202011963@mail.sdu.edu.cn}  
	\and Xue Zhang\textsuperscript{3} 
}

\footnotetext[1]{Zhongtai Securities Institute for Financial Studies,
	Shandong University, Jinan, Shandong 250100, PR China. humingshang@sdu.edu.cn.
	Research supported by the National Natural Science Foundation of China (No. 12326603, 11671231).}
\footnotetext[2]{Zhongtai Securities Institute for Financial Studies,
	Shandong University, Jinan, Shandong 250100, PR China. 202011963@mail.sdu.edu.cn.}
\footnotetext[3]{Department of Mathematics, National University of Defense Technology, Changsha, Hunan 410073, PR China.
	zhangxue\_1998@nudt.edu.cn.}

\maketitle

\textbf{Abstract}. In this paper, we define the squared $ G $-Bessel process as the square of the modulus of a class of $ G $-Brownian motions and establish that it is the unique solution to a stochastic differential equation. We then derive several path properties of the squared $ G $-Bessel process, which are more profound in the capacity sense. Furthermore, we provide upper and lower bounds for the Laplace transform of the squared $ G $-Bessel process. Finally, we prove that the time–space transformed squared $ G $-Bessel process is a $ G' $-CIR process. 



\addcontentsline{toc}{section}{\hspace*{1.8em}Abstract}

\section{Introduction}
In classical probability theory, Bessel processes and some of their generalizations are widely used in financial mathematics. In practical research, the squared Bessel process is not only widely used as a tool for studying the Bessel process, but also serves as a key theoretical reference for analyzing the properties of stochastic processes such as Brownian motion and the Cox–Ingersoll–Ross (CIR) process. The theory and application of the squared Bessel process are described in \cite{G2003survey,jea2009math,brownian2016,countinuous2013,yor2001expon} and their references. 

The problem of uncertainty in volatility in financial markets is complicated to handle within the classical probability framework. To address this issue, Peng \cite{peng2004filtration,peng2007G,peng2008multi,peng2019nonlinear} introduced $ G $-expectation and $ G $-Brownian motion and established the corresponding stochastic calculus theory to incorporate volatility uncertainty into the model. In recent years, the $ G $-expectation theory has attracted widespread attention. For studies on the stochastic differential equations driven by $ G $-Brownian motion ($ G $-SDEs), see \cite{A2023path,gao2009pathwise,hu2021on,liny2013stochastic} and the references therein. A series of results on the L\'{e}vy martingale characterization of $ G $-Brownian motion can be found in \cite{hu2019levy,xu2009martingale,xu2010martingale}. Quasi-continuity for stopping times under G-expectation was established in \cite{song2011hitting,liu2020exit,song2021grad}. Wang and Zheng \cite{wang2018sample} investigated the sample path properties of one-dimensional $ G $-Brownian motion. Akhtari and Li \cite{A2023GCIR} investigated the existence and uniqueness of solutions, the strong Markov property, and other properties of the CIR process under the $ G $-expectation framework ($ G $-CIR process).

This paper aims to study squared Bessel processes under the $ G $-expectation framework (squared $ G $-Bessel processes). In \cite{hu2025gbes}, condition (A) (see Section \ref{sec2}) not only ensures that the quadratic variation process of the $ G $-Brownian motion exhibits desirable properties, but also guarantees that the $ G $-Bessel process is well-defined. Therefore, for similar reasons, we also impose condition (A) on $ G $ in this article. We define the squared $ G $-Bessel process as the square of the modulus of the $ G $-Brownian motion and derive some properties of the squared $ G $-Bessel process directly from the properties of $ G $-Brownian motion. Subsequently, we show that the squared $ G $-Bessel process is the unique non-negative solution to equation (\ref{besq}). 

In the probabilistic framework, the path properties of a stochastic process can be characterized via stopping times. For instance, consider a non-negative random process $ X $ and $ \tau_{1}=\inf\{t\geq0:X_{t}=1\} $. When $ \tau_{1}<\infty $ a.s. holds, it indicates that $ X $ can reach $ 1 $ within a finite time. Through $ \tau_{1}<\infty $ a.s., we can directly obtain $\lim_{t\uparrow\infty}\mathrm{P}(\{\tau_{1}>t\})=\mathrm{P}(\lim_{t\uparrow \infty}\{\tau_{1}>t\})=\mathrm{P}(\{\tau_{1}=\infty\})=0 $. Within the $ G $-expectation framework, we describe the possibility of an event via its capacity. Nevertheless, the equation given above fails to hold in the capacity sense. $ \lim_{t\uparrow\infty}\mathrm{c}(\{\tau_{1}>t\})=0 $ cannot be obtained from $ \mathrm{c}(\{\tau_{1}=\infty\})=0 $ since the capacity does not satisfy upper continuity. But through monotonicity, we can get $ \mathrm{c}(\{\tau_{1}=\infty\})=0 $ from $ \lim_{t\uparrow\infty}\mathrm{c}(\{\tau_{1}>t\})=0 $. This means that  $ \lim_{t\uparrow\infty}\mathrm{c}(\{\tau_{1}>t\})=0 $ is more profound in the capacity sense.

In the study of path properties of the squared $ G $-Bessel process, we need to use the tool of stopping times. To ensure that the stopping time is well-defined within the $ G $-expectation framework, we verified the quasi-continuity of the stopping time. Using a proof by contradiction, we derive the path property of the squared $ G $-Bessel process away from the origin. By introducing stopping times truncation, and constructing two symmetric martingales that take values $ 0 $ or $ 1 $ at the truncation point (see (\ref{phiz}) and (\ref{varphiz})), we obtain the path properties of the square $ G $-Bessel process close to the origin. It is worth pointing out that these results are more profound in the capacity sense. Based on these properties, we can derive an expression for the $ G $-Bessel process in an alternative way, thereby improving the result in \cite{hu2025gbes}.

In addition, we provide upper and lower bounds on the Laplace transform of the squared $ G $-Bessel process. This estimate is different from the classical case but includes it. Finally, by means of the L\'{e}vy martingale characterization of the $ G $-Brownian motion, we obtain the deterministic time transformation formula for one-dimensional $ G' $-Brownian motion ($ G' $ is defined in condition (A)). Using this formulation, we show that the space-time transformed squared $ G $-Bessel process is a one-dimensional $ G' $-CIR process.

This paper is organized as follows. In Section \ref{sec2}, we introduce some fundamental notations and results about $ G $-expectation theory. Section \ref{sec3} presents the definition and path properties of the squared $ G $-Bessel process. Estimates for the Laplace transform of the squared $ G $-Bessel process are provided in Section \ref{sec4}. Finally, we presented the correspondence between the squared $ G $-Bessel process and the $ G' $-CIR process in Section \ref{sec5}.

\section{Preliminaries}\label{sec2}
In this section, we recall the some basic notions and some necessary results of the $G$-expectation framework. The readers may refer to \cite{peng2004filtration,hu2016quasi,hu2019levy,hu2021on,peng2007G,peng2008multi,peng2019nonlinear,denis2011function} for more details.

Let $\Omega =C_{0}^{d}(\mathbb{R}^{+})$ denote the space of all continuous functions $\omega: \mathbb{R}^{+}\rightarrow \mathbb{R}^{d}$ with $\omega _{0}=0$, equipped with the distance 
\begin{equation*}
	\rho (\omega^{1},\omega ^{2}):=\sum_{i=1}^{\infty }2^{-i}[(\max_{t\in
	\lbrack0,i]}|\omega_{t}^{1}-\omega_{t}^{2}|)\wedge1],\quad\omega^{1},\omega^{2}\in \Omega .
\end{equation*}%
Denote the canonical process by $B_{t}(\omega )=\omega _{t},t\in \mathbb{R}%
^{+}$. For each given $t\geq 0$, we define 
\begin{equation*}
	Lip(\Omega _{t}):=\{\varphi (B_{t_{1}\wedge t},\cdots ,B_{t_{n}\wedge
	t}):n\in \mathbb{N},t_{1},\cdots ,t_{n}\in \lbrack 0,\infty ),\varphi\in C_{l.Lip}(\mathbb{R}^{d\times n})\},\ Lip(\Omega ):=\cup _{t=1}^{\infty}Lip(\Omega _{t}),
\end{equation*}%
where $C_{l,Lip}(\mathbb{R}^{d\times n})$ is the space of all local
Lipschitz functions on $\mathbb{R}^{d\times n}$.

Let $\mathbb{S}_{d}$ denote the set of all $d\times d$ symmetric matrices and let $\mathbb{S}_{d}^{+}$ be the subset of non-negative matrices in $\mathbb{S}_{d}$. For each given monotonic and sublinear function $G:\mathbb{S}_{d}\rightarrow \mathbb{R}$, Peng constructed the $G$-expectation $\hat{\mathbb{E}}$ on $(\Omega ,Lip(\Omega ))$ (see \cite{peng2019nonlinear} for definition). Then the canonical process $B$ is a \thinspace $d$-dimensional $G$-Brownian motion under $\hat{\mathbb{E}}$. According to \cite{peng2019nonlinear}, for each $A\in\mathbb{S}_{d}$, there is a bounded, convex and closed subset $\Gamma \subset \mathbb{S}_{d}^{+}$ such that  
\begin{equation*}
	G(A)=\frac{1}{2}\sup_{\gamma\in \Gamma }\mathrm{tr}(\gamma A).
\end{equation*}
In this paper, we assume that there exist two constants $0<\underline{%
	\sigma }\leq \bar{\sigma}<\infty $ such that 
\begin{equation*}
	\frac{1}{2}\underline{\sigma }^{2}\mathrm{tr}[A-B]\leq G(A)-G(B)\leq \frac{1}{2}\bar{\sigma }^{2}\mathrm{tr}[A-B],\quad \text{for all }A\geq B.
\end{equation*}

For each $p\geq 1$, $L_{G}^{p}(\Omega )$ is defined as the completion of $Lip(\Omega )$ under the norm $\Vert X\Vert_{L_{G}^{p}}=(\hat{\mathbb{E}}[|X|^{p}])^{\frac{1}{p}}$. For each $t>0$, $L_{G}^{p}(\Omega _{t})$ can be similarly defined. $\hat{\mathbb{E}}:Lip(\Omega )\rightarrow \mathbb{R}$ can be continuously extended to the mapping from $L_{G}^{1}(\Omega )$ to $%
\mathbb{R}$. $(\Omega ,L_{G}^{1}(\Omega ), \hat{\mathbb{E}})$ is called $G$-expectation space.

Furthermore, we can define the conditional expectation of $ \xi=\varphi(B_{t_{1}},B_{t_{2}}-B_{t_{1}},\cdots,B_{t_{n}}-B_{t_{n-1}}) \in Lip(\Omega) $ through
\[ \hat{\mathbb{E}}_{t_{i}}[\xi]=\phi(B_{t_{1}},\cdots,B_{t_{i}}-B_{t_{i-1}}), \]
where $ \phi(x_{1},\cdots,x_{i})=\hat{\mathbb{E}}[\varphi(x_{1},\cdots,x_{i},B_{t_{i+1}}-B_{t_{i}},\cdots,B_{t_{n}}-B_{t_{n-1}})],\ 0\leq i\leq n $. Similarly, for each fixed $ t \geq 0 $, $ \hat{\mathbb{E}}_{t}:L_{G}^{p}(\Omega)\rightarrow L_{G}^{p}(\Omega_{t}) $ is well-defined.

\begin{theorem}[\cite{hu2009representation,denis2011function}]
	\label{thm2.1} Let $(\Omega ,L_{G}^{1}(\Omega),\hat{\mathbb{E}}) $ be a $G $-expectation space. Then there exists a weakly compact set of probability measures $\mathcal{P} $ on $(\Omega,\mathcal{F}) $ such that 
	\begin{align}
		\hat{\mathbb{E}}[\xi]=\sup_{P\in\mathcal{P}}E_{P}[\xi], \ \text{ for each }\xi\in L_{G}^{1}(\Omega).
	\end{align}
\end{theorem}

For this $\mathcal{P}$, we can define the capacity 
\begin{equation*}
	\mathrm{c}(A):=\sup_{P\in \mathcal{P}}P(A),\ \text{ for each }A\in \mathcal{F},
\end{equation*}
where $\mathcal{F}:=\bigvee_{t\geq 0}\mathcal{F}_{t}$ and $\mathcal{F}_{t}:=\sigma (B_{s}:s\leq t)$.

\begin{definition}
	A set $A\in \mathcal{F}$ is polar if $\mathrm{c}(A)=0$. Moreover, a property holds \textquotedblleft quasi-surely'' (q.s.) if it holds outside a polar set.
\end{definition}

\begin{definition}
	For each $T>0$ and $p\geq 1$, set 
	\begin{equation*}
		M_{G}^{p,0}(0,T):=\left\{ \eta (t)=\sum_{j=0}^{n-1}\xi
		_{j}I_{[t_{j},t_{j+1})}(t):n\in \mathbb{N},0=t_{0}<t_{1}<\cdots <t_{n}=T,\ \xi _{j}\in L_{G}^{p}(\Omega _{t_{j}})\right\} .
	\end{equation*}%
	$M_{G}^{p}(0,T)$ is defined as the completion of $M_{G}^{p,0}(0,T)$ under the norm $\Vert\eta\Vert_{M_{G}^{p}}:=(\hat{\mathbb{E}}[\int_{0}^{T}|\eta
	(t)|^{p}dt])^{\frac{1}{p}}$.
\end{definition}

According to \cite{peng2019nonlinear}, the integrals $\int_{0}^{t}\eta_{s}\mathrm{d}B_{s}$ and $\int_{0}^{t}\mu_{s} \mathrm{d}\langle B\rangle _{s}$ are well-defined for $\eta \in M_{G}^{2}(0,T)$ and $\mu \in M_{G}^{1}(0,T)$, where $\langle B\rangle $ denotes the quadratic variation process of $B$.

\begin{theorem}[\cite{hu2019levy}]\label{Levy}
	Let $ G $ be a given monotonic and sublinear function and let $(\Omega ,L_{G}^{1}(\Omega),\hat{\mathbb{E}}) $ be a sublinear expectation space. Suppose that $ (M_{t})_{t\geq0} $ is a $ d $-dimensional symmetric martingale satisfying $M_{0}=0,M_{t}\in(L_{G}^{3}(\Omega_{t}))^{d} $ for each $ t\geq0 $. In addition, we also need to assume that $\sup\{\hat{\mathbb{E}}[|M_{t+\epsilon}-M_{t}|^{3}]:t\leq T\}=o(\epsilon)$ as $ \epsilon\downarrow0 $ for each $ T\geq0 $. If the process $ \frac{1}{2}\left\langle AM_{t},M_{t}\right\rangle-G(A)t ,t\geq0 $, is a martingale for each $ A\in\mathbb{S}_{d} $, then $ M $ is a $ G $-Brownian motion.
\end{theorem}

\begin{remark}
	Let  $ d=1 $ and let $ G(x)=\frac{1}{2}(\bar{\sigma}^{2}x^{+}-\underline{\sigma}^{2}x^{-}) $ where $ 0<\underline{\sigma}\leq\bar{\sigma}<\infty $. Let $ (M_{t})_{t\geq0} $ be a symmetric martingale satisfying $ M_{0}=0 , M_{t}\in L_{G}^{3}(\Omega_{t}) $ for each $ t\geq0 $.  If $ \{M_{t}^{2}-\bar{\sigma}^{2}t\}_{t\geq0}  ,\  \{-M_{t}^{2}+\underline{\sigma}^{2}t\}_{t\geq0} $ are martingales and $\sup\{\hat{\mathbb{E}}[|M_{t+\epsilon}-M_{t}|^{3}]:t\leq T\}=o(\epsilon)$ as $ \epsilon\downarrow0 $ for each $ T\geq0 $, then $ M $ is a one-dimensional $ G $-Brownian motion.
\end{remark}

Let $d\geq 1$ and let $(\Omega ,L_{G}^{1}(\Omega ),\hat{\mathbb{E}})$ be a $
G $-expectation space. Let $B=(B^{(1)},...,B^{(d)})$ be a $d$-dimensional $G$-Brownian motion. For each $x\in^{d}$, we denote by $B^{x}:=B+x=(\bar{B}^{(1)},...,\bar{B}^{(d)})$ the $d$-dimensional $G$-Brownian motion starting at $x$. Obviously, $\mathrm{d}\bar{B}_{t}^{(i)}= \mathrm{d}B_{t}^{(i)},\ \mathrm{d}\langle B^{(i)},B^{(j)}\rangle_{t}=\mathrm{d}\langle \bar{B}^{(i)},\bar{B}^{(j)}\rangle_{t} $ for $ 1\leq i,j\leq d$. For $t\geq0$, define $\beta_{t}=\sum_{i=1}^{d}\int_{0}^{t}\frac{\bar{B}_{s}^{(i)}}{\left\vert
B_{s}^{x}\right\vert }\mathrm{d}B_{s}^{(i)}$. Set $ \underline{\sigma }^{2}=-\hat{\mathbb{E}}\left[ -\langle \beta \rangle _{1}\right],\ \bar{\sigma}^{2}=\hat{\mathbb{E}}\left[ \langle \beta \rangle_{1}\right]  $.

\begin{description}
	\item[(A)] For each $A\in \mathbb{S}_{d},\ G(A)=G'(\mathrm{tr}[A])$ where $G'(a)=\frac{1}{2}(\bar{\sigma}^{2}a^{+}-\underline{\sigma }^{2}a^{-})$ for $a\in \mathbb{R} $.
\end{description}

\begin{lemma}[\cite{hu2025gbes}]\label{beta}
	Suppose $G$ satisfies condition (A). Then for $1\leq
	i,j\leq d$ and $i\neq j$, we have $\langle B^{(1)}\rangle =\langle B^{(i)}\rangle,\langle B^{(i)},B^{(j)}\rangle =0$ q.s. In this case, $\beta $ is a one-dimensional $G'$-Brownian motion and $\langle \beta \rangle=\langle B^{(1)}\rangle $ q.s. Moreover, for $ 0\leq s\leq t $, $ \beta_{t}-\beta_{s} $ is independent from $ L_{G}^{1}(\Omega_{s}) $.
\end{lemma}

\begin{proof}
	The default condition in \cite{hu2025gbes} is $ d\geq2 $, but the above lemma still holds when $ d=1 $. Since the proof of Lemma 4.3 in \cite{hu2025gbes} also works to $ d=1 $, we omit it here. Furthermore, for the case when $ d=1 $, one can also refer to Lemma 5.4 in \cite{hu2019levy}.
\end{proof}

\begin{definition}[\cite{hu2025gbes}]\label{def bes}
	Let $ G $ be a function satisfying the condition (A). For $ x\in\mathbb{R}^{d} $,  let $ B^{x} $ be a $ d $-dimensional $ G $-Brownian motion starting at $ x $ defined as above. 
	\[R_{t}:=|B_{t}^{x}|=\sqrt{\left( \bar{B}_{t}^{(1)}\right) ^{2}+\cdots+\left( \bar{B}_{t}^{(d)}\right) ^{2}}, \quad t\in[0,\infty) , \]
	is called a $ G $-Bessel process with dimension $ d $ starting at $ r=|x| $ and is denoted by $ G $-$\mathrm{BES} _{\mathrm{r}}^{\mathrm{d}} $.
\end{definition}

\section{Definition and path properties of squared $G$-Bessel processes}\label{sec3}

According to Definition \ref{def bes}, we can define the squared $G$-Bessel process in a similar way.

\begin{definition}[squared $G$-Bessel processes]\label{def besq}
	Let $d\geq 1$ and let $G$ satisfy the condition (A). For $x\in \mathbb{R}^{d}$, let $B^{x}:=B+x=(\bar{B}^{(1)},...,\bar{B}^{(d)})$ be a $d$-dimensional $G$-Brownian motion starting at $x$ on $(\Omega,L_{G}^{1}(\Omega),\hat{\mathbb{E}})$. 
	\begin{equation*}
		Z_{t}:=\left\vert B_{t}^{x}\right\vert^{2}=\left( \bar{B}_{t}^{(1)}\right) ^{2}+...+\left( \bar{B}_{t}^{(d)}\right) ^{2},\quad t\geq 0,
	\end{equation*}
	is called a squared $G$-Bessel process with dimension $d$ starting at $z=\left\vert x\right\vert ^{2}$ and is denoted by $G$-$\mathrm{BESQ}_{z}^{d}$.
\end{definition}

\begin{remark}
	From Lemma 3.3 of \cite{hu2025gbes}, we know that $ |B^{y}|^{2}\overset{d}{=}|B^{x}|^{2} $ when $ |y|^{2}=|x|^{2} $. So the squared $G$-Bessel process is well-defined.
\end{remark}

Since the squared $G$-Bessel process is defined by $B^{x}$, some properties
can be obtained through the properties of $B^{x}$.

\begin{corollary}\label{cor3.3}
	For any $\lambda >0$, if $Z$ is a $G$-$\mathrm{BESQ}_{z}^{d}$, then
	process $(\lambda ^{-1}Z_{\lambda t})$ is a $G$-$\mathrm{BESQ}_{z\lambda ^{-1}}^{d}$.
\end{corollary}

\begin{proof}
	Through a change of the variable, we have
	\begin{equation*}
		\lambda ^{-1}Z_{\lambda t}=\left( \lambda ^{-\frac{1}{2}}\bar{B}_{\lambda
		t}^{(1)}\right) ^{2}+...+\left( \lambda ^{-\frac{1}{2}}\bar{B}_{\lambda t}^{(d)}\right) ^{2},\quad t\geq 0.
	\end{equation*}
	From Remark 1.4 in Chapter III of \cite{peng2019nonlinear}, $(\lambda ^{-\frac{1}{2}}B_{\lambda t}^{x})_{t\geq 0}$ is a $d$-dimensional $G$-Brownian motion starting at $\lambda ^{-\frac{1}{2}}x$. Thus, we can draw the conclusion.
\end{proof}

\begin{corollary}\label{cor3.4}
	If $Z$ is a $G$-$\mathrm{BESQ}_{z}^{d}$, then 
	\begin{equation*}
		\lim_{t\rightarrow \infty }\frac{Z_{t}}{t^{2}}=0\quad \text{q.s.}
	\end{equation*}
\end{corollary}

\begin{proof}
	Without loss of generality, we may assume $ z=0 $. According to the definition \ref{def besq}, we only need to prove that $%
	\lim_{t\rightarrow \infty }\frac{B_{t}^{(1)}}{t}=0 $ q.s. Let $A_{\delta
	}=\{\omega :\limsup_{t\rightarrow \infty }|\frac{B_{t}^{(1)}(\omega )}{t}|>\delta \}$ for any $\delta >0$. Define 
	\begin{equation*}
		B_{\delta }^{n}=\left\lbrace \omega :\sup_{2^{n-1}\leq t\leq 2^{n}}\left| B_{t}^{(1)}(\omega)\right| >\delta 2^{n-1}\right\rbrace .
	\end{equation*}
	Clearly, $A_{\delta }\subset\limsup_{n\rightarrow \infty } B_{\delta }^{n}$. By Corollary 4.11 of \cite{xu2010martingale}, we obtain 
	\begin{align}
		\mathrm{c}(B_{\delta }^{n})\leq {}& \mathrm{c}\left( \left\lbrace \sup_{0\leq t\leq2^{n}}\left| B_{t}^{(1)}\right| >\delta 2^{n-1}\right\rbrace \right)   \notag \\
		\leq {}& \frac{\hat{\mathbb{E}}\left[ \left| B_{2^{n}}^{(1)}|\right| ^{2}\right] }{(\delta 2^{n-1})^{2}}\notag \\
		={}& \frac{\bar{\sigma}^{2}}{\delta ^{2}2^{n-2}}.  \notag
	\end{align}%
	Since $\sum_{n=1}^{\infty }\mathrm{c}(B_{\delta }^{n})<\infty $, by Lemma 5
	of \cite{denis2011function}, it follows that $\mathrm{c}(\limsup_{n
		\rightarrow \infty }B_{\delta }^{n})=0$. Therefore, $A_{\delta }$ is polar.
	Since $\delta $ can be an arbitrary positive number, we get the desired
	result.
\end{proof}

\begin{corollary}
	If $ Z $ is a $G$-$\mathrm{BESQ}_{0}^{d}$. Define  
	\begin{align*}
		X_{t}=
		\begin{cases}
			t^{2}Z_{\frac{1}{t}},\quad t>0,\\
			0,\quad t=0.
		\end{cases}
	\end{align*}
	Then for each $ t\geq0 $, $ X_{t}\overset{d}{=}Z_{t} $. Moreover, $ X $ is continuous in $ t $. 
\end{corollary}

\begin{proof}
	Combining Corollary \ref{cor3.3} and Corollary \ref{cor3.4} gives us the desired result.
\end{proof}

We now present the G-SDE satisfied by the squared $G$-Bessel process.

\begin{proposition}\label{solution besq}
	Let $T,z\geq0$ and let $d\in \mathbb{N}^{+}$. Let $Z$ be the $G$-$\mathrm{BESQ}_{z}^{d}$. Then $Z$ is the unique non-negative solution of the following equation in $ M_{G}^{1}(0,T) $:
	\begin{equation}\label{besq}
		Z_{t}=z+2\int_{0}^{t}\sqrt{Z_{s}}\mathrm{d}\beta _{s}+d\left\langle \beta\right\rangle _{t},\quad t\in \left[ 0,T\right].
	\end{equation}
	Here $\beta _{t}:=\sum_{i=1}^{d}\int_{0}^{t}\frac{\bar{B}_{s}^{(i)}}{
	\left\vert B_{s}^{x}\right\vert }\mathrm{d}B_{s}^{(i)},t\in \lbrack 0,T]$, is a one-dimensional $G'$-Brownian motion. Moreover, $Z\in \tilde{M}_{G}^{2}(0,T)$.
\end{proposition}

\begin{remark}
	As $Z$ is generated by $B^{x}=B+x$, the solution space for the equation is restricted to $M_{G}^{1}(0,T)$. Since $\beta $ is a one-dimensional $
	G'$-Brownian motion, we can define $\tilde{L}_{G}^{p}(\Omega ),\ \tilde{M}_{G}^{p}(0,T)$ in a similar way. From the fact that $\beta $ is
	generated by $B^{x}$, it follows that $\tilde{L}_{G}^{p}(\Omega )\subset
	L_{G}^{p}(\Omega ),\ \tilde{M}_{G}^{p}(0,T)\subset M_{G}^{p}(0,T)$.
\end{remark}

\begin{proof}
	Applying the It\^{o} formula to $Z_{t}$ on $ [0,T] $, we have
	\begin{equation}\label{Bx}
		Z_{t}=\left\vert B_{t}^{x}\right\vert ^{2}=\left\vert x\right\vert
		^{2}+2\sum_{i=1}^{d}\int_{0}^{t}\bar{B}_{s}^{(i)}\mathrm{d}B_{s}^{(i)}+\sum_{i=1}^{d}\langle B^{(i)}\rangle _{t},\quad t\in \left[ 0,T\right] .
	\end{equation}%
	According to Lemma \ref{beta}, $\beta $ is a one-dimensional $G'$-Brownian motion and $\langle \beta \rangle =\langle B^{(1)}\rangle $ q.s. Then equation (\ref{Bx}) turns to
	\begin{equation*}
		Z_{t}=z+2\int_{0}^{t}\sqrt{Z_{s}}\mathrm{d}\beta _{s}+d\left\langle \beta\right\rangle _{t},\quad t\in \left[ 0,T\right],
	\end{equation*}
	which implies that $ Z $ is the solution to equation (\ref{besq}). Let us now prove the uniqueness. Assume that $ Z'\in M_{G}^{1}[0,T] $ is another solution of equation (\ref{besq}). Note that $\beta _{t}-\beta _{s}$ is independent from $L_{G}^{1}(\Omega_{s}) $ for $0\leq s\leq t\leq T$. According to Proposition 3.2 of \cite{A2023GCIR}, we obtain $ Z=Z' $ q.s. From Theorem 3.6 of \cite{A2023GCIR}, we obtain that above equation has a solution belongs to $\tilde{M}_{G}^{2}(0,T)$. Due to $\tilde{M}_{G}^{2}(0,T)\subset M_{G}^{2}(0,T)\subset M_{G}^{1}(0,T)$, it follows that $Z\in \tilde{M}_{G}^{2}(0,T)$.
\end{proof}

The following lemma is the key for us to  study the properties of the squared $G$-Bessel process further.

\begin{lemma}\label{taub}
	Let $z>0$ and let $d\in \mathbb{N}^{+}$. Let $Z$ be the $G$-$\mathrm{BESQ}_{z}^{d}$. For any $b>z$, set $\tau _{b}=\inf \left\{ t\geq 0:Z_{t}\geq b\right\} $. Then for each fixed $ b $, $\lim_{t\uparrow\infty}\mathrm{c}(\{\tau _{b}>t\})=0 $. Moreover, $\lim_{b\uparrow\infty}\mathrm{c}(\{\tau _{b}<t\})=0 $  for $ t\geq0 $.
\end{lemma}

\begin{proof}
	From Lemma 4.3 of \cite{song2021grad}, we get $\tau _{b}\wedge t$ is a
	quasi-continuous random variable. According to equation (\ref{besq}), we have
	\begin{equation*}
		Z_{\tau _{b}\wedge t}=z+2\int_{0}^{\tau _{b}\wedge t}\sqrt{Z_{s}}\mathrm{d}\beta_{s}+d\left\langle \beta \right\rangle _{\tau _{b}\wedge t}, \quad t\geq0.
	\end{equation*}%
	Denote $Y_{\tau _{b}\wedge t}=\int_{0}^{\tau _{b}\wedge t}2\sqrt{Z_{s}}%
	\mathrm{d}\beta _{s}$. By Corollary 4.10 of \cite{liu2020exit}, we know that $Y_{\tau_{b}\wedge t},t\geq 0$ is a symmetric martingale. From the definition of $\tau _{b}$, we obtain
	\begin{align*}
		\hat{\mathbb{E}}\left[ \left\vert Y_{\tau _{b}\wedge t}\right\vert ^{2}\right] ={}&\hat{\mathbb{E}}\left[ \int_{0}^{\tau _{b}\wedge
			t}4Z_{s}\mathrm{d}\left\langle \beta \right\rangle _{s}\right]\\
		 \leq{}& 4b\bar{\sigma}^{2}t.
	\end{align*}
	According to Lemma 13 of \cite{denis2011function}, we get
	\begin{align}\label{Ftc}
		\mathrm{c}\left( \left\{  |Y_{\tau _{b}\wedge t}| >t^{\frac{3}{4}}\right\} \right) \leq{}& \frac{\hat{\mathbb{E}}\left[ \left\vert Y_{\tau _{b}\wedge t}\right\vert ^{2}\right] }{t^{\frac{3}{2}}}\notag\\
		\leq{}& \frac{4b\bar{\sigma}^{2}}{\sqrt{t}}\rightarrow 0 \quad \text{ as } t\rightarrow \infty.
	\end{align}
    Here, we employ a proof by contradiction to establish that $ \lim_{t\uparrow\infty}\mathrm{c}(\{\tau _{b}>t\})=0 $. Suppose there exists $ \epsilon>0 $ such that, for any $ n\in \mathbb{N}^{+} $, $ \mathrm{c}(\{\tau _{b}>n\})\geq\epsilon $. Then for each $ \omega\in \{\tau _{b}>n\} $, $\tau _{b}(\omega)\wedge n=n$. Define $ F_{n}=\{ \omega:|Y_{\tau _{b}(\omega)\wedge n}|\leq n^{\frac{3}{4}} \} $. From (\ref{Ftc}), we can find $ n' $ such that $ \mathrm{c}(F_{n}^{c})\leq \frac{\epsilon}{2} $ whenever $ n\geq n' $. By subadditivity and monotonicity of the capacity, we obtain 
    \begin{align*}
    	\mathrm{c}(\{\tau _{b}>n\})\leq{}&\mathrm{c}(\{\tau _{b}>n\}\cap F_{n})+\mathrm{c}(\{\tau _{b}>n\}\cap F_{n}^{c})\notag \\
    	\leq{}&\mathrm{c}(\{\tau _{b}>n\}\cap F_{n})+\frac{\epsilon}{2}.
    \end{align*}
    Due to $ \mathrm{c}(\{\tau _{b}>n\})\geq\epsilon $, we get $ \mathrm{c}(\{\tau _{b}>n\}\cap F_{n})\geq \frac{\epsilon}{2} $. For each $ \omega\in \{\tau _{b}>n\}\cap F_{n} $, we have
	\begin{equation*}
		b>Z_{n}(\omega)=z+Y_{n}(\omega)+d\left\langle \beta \right\rangle _{n}(\omega)\geq z+d\underline{\sigma}^{2}n-n^{\frac{3}{4}}.
	\end{equation*}
    For each given $ b $, choose $ m\geq n' $ such that $ z+d\bar{\sigma}^{2}m-m^{\frac{3}{4}}>b $, then $ \{b>z+d\underline{\sigma}^{2}m-m^{\frac{3}{4}} \}=\emptyset $.
    Therefore, 
    \begin{align*}
    	\mathrm{c}(\{\tau _{b}>m\}\cap F_{m})\leq\mathrm{c}(\{b>z+d\underline{\sigma}^{2}m-m^{\frac{3}{4}} \})=0 .
    \end{align*}
	This gives a contradiction. Thus we obtain $\lim_{t\uparrow\infty}\mathrm{c}(\{\tau _{b}>t\})=0 $.
	
	Next, we show that $\lim_{b\uparrow\infty}\mathrm{c}(\{\tau _{b}<t\})=0 $. We proceed with the proof by contradiction. Obviously $\tau _{b}$ is nondecreasing. Suppose there exists $t_{0},\epsilon_{0}>0 $ such that, for any $ b $, $ \mathrm{c}\left( \left\lbrace \tau _{b}<t_{0}\right\rbrace \right) \geq \epsilon_{0} $. For each $ \omega\in \left\lbrace \tau _{b}<t_{0}\right\rbrace $,
	\begin{align*}
		b=Z_{\tau_{b}\wedge t_{0}}(\omega )={}&z+Y_{\tau_{b}}(\omega )+d\left\langle \beta\right\rangle _{\tau_{b}}(\omega )\\
		\leq{}&z+\sup_{t\in[0,t_{0}]}|Y_{t}|(\omega)+d\bar{\sigma}^{2}t_{0}.
	\end{align*}
    Therefore, 
    \begin{align*}
    	\mathrm{c}\left( \left\lbrace  \sup_{t\in[0,t_{0}]}|Y_{t}|\geq b-z-d\bar{\sigma}^{2}t_{0} \right\rbrace \right) \geq\mathrm{c}\left( \left\lbrace \tau _{b}<t_{0}\right\rbrace \right) \geq \epsilon_{0}.
    \end{align*}
    According to (\ref{besq}), we yield
    \begin{align}\label{yt0}
    	\hat{\mathbb{E}}\left[ Y_{t_{0}}^{2}\right]\leq{}&4\bar{\sigma}^{2}\hat{\mathbb{E}}\left[ \int_{0}^{t_{0}}Z_{s}\mathrm{d}s\right] \notag\\
    	\leq{}&4\bar{\sigma}^{2}\int_{0}^{t_{0}}\hat{\mathbb{E}}\left[ Z_{s}\right]\mathrm{d}s\notag\\
    	\leq{}&4\bar{\sigma}^{2}t_{0}\left( z+d\bar{\sigma}^{2}t_{0}\right).
    \end{align}
    From Corollary 4.11 of \cite{xu2010martingale} and (\ref{yt0}), it follows that
    \begin{align*}
    	\mathrm{c}\left( \left\lbrace  \sup_{t\in[0,t_{0}]}|Y_{t}|\geq b-z-d\bar{\sigma}^{2}t_{0} \right\rbrace \right)
    	\leq {}&\frac{\hat{\mathbb{E}}\left[Y_{t_{0}}^{2} \right] }{\left(b-z-d\bar{\sigma}^{2}t_{0} \right)^{2} }\\
    	\leq{}&\frac{4\bar{\sigma}^{2}t_{0}\left(z+d\bar{\sigma}^{2}t_{0} \right) }{\left(b-z-d\bar{\sigma}^{2}t_{0} \right)^{2} }\rightarrow0\quad \text{ as }b\rightarrow\infty.
    \end{align*}
	This presents a contradiction, and thus we conclude that $ \lim_{b\uparrow\infty}\mathrm{c}(\{\tau _{b}<t\})=0  $.
\end{proof}

\begin{remark}\label{remtaub}
	 It should be noted that, in the capacity sense, the conclusion of the above lemma is not equivalent to $ \tau_{b}<\infty $ q.s. and $ \lim_{b\uparrow\infty}\tau_{b}=\infty $ q.s. Strictly speaking, the conclusion of the above lemma can lead to $ \tau_{b}<\infty $ q.s. and $ \lim_{b\uparrow\infty}\tau_{b}=\infty $ q.s., but the converse implication does not hold. For example,  
	\[ \mathrm{c}\left(\left\lbrace \tau_{b}=\infty \right\rbrace \right)\leq \mathrm{c}\left(\left\lbrace \tau_{b}>t \right\rbrace \right). \]
	Then, taking the limit of both sides with respect to $ t $, we have 
	\[ \mathrm{c}\left(\left\lbrace \tau_{b}=\infty \right\rbrace \right)\leq\lim_{t\uparrow \infty }\mathrm{c}\left(\left\lbrace \tau_{b}>t \right\rbrace \right)=0. \]
	Thus we get $ \tau_{b}<\infty $ q.s. However, since capacity differs from probability and does not satisfy upper continuity, we cannot obtain $ \lim_{t\uparrow \infty }\mathrm{c}\left(\left\lbrace \tau_{b}>t \right\rbrace \right)=0 $ through $ \tau_{b}<\infty $ q.s. 
\end{remark}

Here we present the main result of this section.

\begin{theorem}\label{path}
	Let $z>0$ and let $d\in \mathbb{N}^{+}$. Let $Z$ be the $G$-$\mathrm{BESQ}_{z}^{d}$. For each $a\in (0,z)$, define $\tau _{a}=\inf \left\{ t\geq0:Z_{t}\leq a\right\} $ and $ \tau_{0}=\lim_{a\downarrow0}\tau_{a} $. Then 
	\begin{itemize}
		\item[(1)] If $d\geq 2$, $\lim_{a\downarrow0}\mathrm{c}\left( \left\{ \tau _{a}<t\right\} \right)=0$ for $ t\geq0 $;
		\item[(2)] If $d=2$, $\lim_{t\uparrow\infty}\mathrm{c}\left( \left\{ \tau_{a}>t \right\} \right)=0$;
		\item[(3)] If $d=1$, $\lim_{t\uparrow\infty}\mathrm{c}\left( \left\{ \tau_{0}>t \right\} \right)=0$;
		\item[(4)] If $d>2$, $\mathrm{c}\left( \left\{ \inf_{t\geq 0}Z_{t}=0\right\} \right)=0$.
	\end{itemize}
\end{theorem}

\begin{proof}
	Obviously, when $ t\leq\tau_{a} $, $ (Z_{t})_{t\leq\tau_{a}} $ satisfies the condition ($ H' $) in Remark 3.5 of \cite{liu2020exit}. According to Corollary 3.7 of \cite{liu2020exit}, for $ t\geq0 $, $\tau _{a}\wedge t$ is a quasi-continuous random variable. $ \tau_{b} $ is defined as in Lemma \ref{taub}. Similarly, we get $ t\wedge\tau_{a}\wedge\tau_{b} $ is a quasi-continuous random variable. Thus, 
	\begin{align}
			Z_{t\wedge\tau_{a}\wedge\tau_{b}}={}&z+2\int_{0}^{t\wedge\tau_{a}\wedge\tau_{b}}\sqrt{Z_{s}}\mathrm{d}\beta_{s}+d\left\langle \beta\right\rangle _{t\wedge\tau_{a}\wedge\tau_{b}}\notag\\ 
			={}&z+2\int_{0}^{t}\sqrt{Z_{s}}I_{[0,\tau_{a}\wedge\tau_{b}]}(s)\mathrm{d}\beta_{s}+d\int_{0}^{t}I_{[0,\tau_{a}\wedge\tau_{b}]}(s)\mathrm{d}\left\langle \beta\right\rangle _{s},\quad t\in[0,T].
	\end{align}
	
	For $ y\in[a,b] $, set
	\begin{align*}
		\phi(y)=
		\begin{cases}
			\frac{y^{1-\frac{d}{2}}-a^{1-\frac{d}{2}}}{b^{1-\frac{d}{2}}-a^{1-\frac{d}{2}}},\quad d\neq 2,\\
			\frac{\ln ^{y}-\ln ^{a}}{\ln ^{b}-\ln ^{a}},\quad d=2,
		\end{cases}
	\end{align*}
	and
	\begin{align*}
		\varphi(y)=
		\begin{cases}
			\frac{b^{1-\frac{d}{2}}-y^{1-\frac{d}{2}}}{b^{1-\frac{d}{2}}-a^{1-\frac{d}{2}}},\quad d\neq 2,\\
			\frac{\ln ^{b}-\ln ^{y}}{\ln ^{b}-\ln ^{a}},\quad d=2.
		\end{cases}
	\end{align*}
	Applying the It\^{o} formula to $ \phi(Z_{t\wedge\tau_{a}\wedge\tau_{b}}) $, we get
	\begin{align}\label{phiz}
		\phi(Z_{t\wedge\tau_{a}\wedge\tau_{b}})=
		\begin{cases}
			\frac{z^{1-\frac{d}{2}}-a^{1-\frac{d}{2}}}{b^{1-\frac{d}{2}}-a^{1-\frac{d}{2}}}+\frac{(2-d)}{b^{1-\frac{d}{2}}-a^{1-\frac{d}{2}}}\int_{0}^{t}Z_{s}^{\frac{1-d}{2}}I_{[0,\tau_{a}\wedge\tau_{b}]}(s)\mathrm{d}\beta_{s},\quad d\neq2,\\
			\frac{\ln^{z}-\ln ^{a}}{\ln ^{b}-\ln ^{a}}+\frac{2}{\ln ^{b}-\ln ^{a}}\int_{0}^{t}Z_{s}^{-\frac{1}{2}}I_{[0,\tau_{a}\wedge\tau_{b}]}(s)\mathrm{d}\beta_{s},\quad d=2.
		\end{cases}
	\end{align}
	Observe that $ \varphi=1-\phi $. Thus we yield
	\begin{align}\label{varphiz}
		\varphi(Z_{t\wedge\tau_{a}\wedge\tau_{b}})=
		\begin{cases}
			\frac{b^{1-\frac{d}{2}}-z^{1-\frac{d}{2}}}{b^{1-\frac{d}{2}}-a^{1-\frac{d}{2}}}+\frac{(d-2)}{b^{1-\frac{d}{2}}-a^{1-\frac{d}{2}}}\int_{0}^{t}Z_{s}^{\frac{1-d}{2}}I_{[0,\tau_{a}\wedge\tau_{b}]}(s)\mathrm{d}\beta_{s},\quad d\neq2,\\
			\frac{\ln^{b}-\ln ^{z}}{\ln ^{b}-\ln ^{a}}-\frac{2}{\ln ^{b}-\ln ^{a}}\int_{0}^{t}Z_{s}^{-\frac{1}{2}}I_{[0,\tau_{a}\wedge\tau_{b}]}(s)\mathrm{d}\beta_{s},\quad d=2.
		\end{cases}
	\end{align}
    Equations (\ref{phiz}) and (\ref{varphiz}) indicate that $ \phi(Z_{t\wedge\tau_{a}\wedge\tau_{b}}),\ t\geq0 $ and $ \varphi(Z_{t\wedge\tau_{a}\wedge\tau_{b}}),\ t\geq0 $ are symmetric martingales. 
    
    Notice that $ \phi (Z_{t\wedge \tau _{a}\wedge \tau _{b}})\in[0,1] $. According to Lemma \ref{taub}, we have
    \begin{align*}
    	{}&\lim_{t\uparrow \infty}\hat{\mathbb{E}}\left[\left| \phi (Z_{t\wedge \tau _{a}\wedge \tau _{b}})-\phi (Z_{t\wedge \tau _{a}\wedge \tau _{b}})I_{\{ \tau_{b}\leq t \}}\right|  \right]\\
    	\leq{}& \lim_{t\uparrow \infty}\hat{\mathbb{E}}\left[\phi (Z_{t\wedge \tau _{a}\wedge \tau _{b}})I_{\{ \tau_{b}>t \}} \right]\notag\\
    	\leq{}&\lim_{t\uparrow \infty}\mathrm{c}(\{\tau _{b}>t\})\\
    	={}&0, 
    \end{align*}
    which implies $ \lim_{t\uparrow \infty}\hat{\mathbb{E}}\left[\phi (Z_{t\wedge \tau _{a}\wedge \tau _{b}}) \right]=\lim_{t\uparrow \infty}\hat{\mathbb{E}}\left[\phi (Z_{t\wedge \tau _{a}\wedge \tau _{b}})I_{\{ \tau_{b}\leq t \}} \right] $. 
    Then we obtain
	\begin{align*}
		\lim_{t\uparrow \infty}\hat{\mathbb{E}}\left[ \phi (Z_{t\wedge \tau _{a}\wedge \tau _{b}})\right] ={}&\lim_{t\uparrow \infty}\hat{\mathbb{E}}\left[ \phi (Z_{t\wedge \tau _{a}\wedge \tau _{b}})I_{\{ \tau_{b}\leq t \}}\right]\notag \\
		={}&\lim_{t\uparrow \infty}\hat{\mathbb{E}}\left[ \phi (Z_{\tau _{a}\wedge \tau _{b}})I_{\{\tau_{b}\leq t \}}\right]\notag\\
		={}&\lim_{t\uparrow \infty}\hat{\mathbb{E}}\left[\phi (Z_{\tau _{a}})I_{\{\tau_{a}\geq\tau_{b}\}}I_{\{\tau_{b}\leq t \}}+\phi (Z_{\tau _{b}})I_{\{\tau_{b}<\tau_{a}\}}I_{\{\tau_{b}\leq t \}}   \right]  \\
		={}&\lim_{t\uparrow \infty}\mathrm{c}\left( \left\{ \tau _{b}<\tau _{a},\tau_{b}\leq t\right\} \right)\\
		={}&\mathrm{c}\left( \left\{ \tau _{b}<\tau _{a}\right\} \right).
	\end{align*}
	By the property of martingale, we get 
	\begin{equation*}
		\phi (z)=\hat{\mathbb{E}}\left[ \phi (Z_{0})\right] =\hat{\mathbb{E}}\left[ \phi (Z_{t\wedge\tau _{a}\wedge \tau _{b}})\right] .
	\end{equation*}
	Taking the limit of $ t $ on both sides, we yield
	\begin{equation*}
		\phi (z)=\lim_{t\uparrow \infty}\hat{\mathbb{E}}\left[ \phi (Z_{t\wedge \tau _{a}\wedge \tau _{b}})\right]=\mathrm{c}\left( \left\{\tau _{b}<\tau _{a}\right\} \right) .
	\end{equation*}
	Therefore,
	\begin{align}\label{b<a}
		\mathrm{c}\left( \left\{ \tau _{b}<\tau _{a}\right\} \right) =
		\begin{cases}
			\frac{z^{1-\frac{d}{2}}-a^{1-\frac{d}{2}}}{b^{1-\frac{d}{2}}-a ^{1-\frac{d}{2}}},\quad d\neq 2,\\
			\frac{\ln^{z}-\ln ^{a}}{\ln ^{b}-\ln ^{a}},\quad d=2.
		\end{cases}
	\end{align}%
    Similarly, it follows that
    \begin{align}\label{a<b}
    	\mathrm{c}\left( \left\{ \tau _{a}<\tau _{b}\right\} \right) =
    	\begin{cases}
    		\frac{b^{1-\frac{d}{2}}-z^{1-\frac{d}{2}}}{b^{1-\frac{d}{2}}-a ^{1-\frac{d}{2}}},\quad d\neq 2,\\
    		\frac{\ln^{b}-\ln ^{z}}{\ln ^{b}-\ln ^{a}},\quad d=2.
    	\end{cases}
    \end{align}

	In the following, we proceed to prove four conclusions of this theorem separately.
	
	(1): Let us first prove that when $ d=2 $, $ \lim_{a\downarrow0}\mathrm{c}\left( \left\{ \tau _{a}<t\right\} \right) =0 $ for $ t\geq0 $. Consider that the constant $ k $ satisfies $ k^{-k}<z$ and $k>z $. Then we can take $ a=k^{-k},b=k $. Based on subadditivity and monotonicity of the capacity and (\ref{a<b}), we have
	\begin{align}
		 \lim_{a\downarrow0}\mathrm{c}\left( \left\{ \tau _{a}<t\right\} \right)\leq{}&\lim_{a\downarrow0}\mathrm{c}\left( \left\{ \tau _{a}<t,\tau_{a}<\tau_{b}\right\}\right)+\lim_{a\downarrow0}\mathrm{c}\left( \left\{ \tau _{a}<t,\tau_{a}\geq\tau_{b}\right\} \right)\notag\\
		 \leq{}&\lim_{a\downarrow0}\mathrm{c}\left( \left\{ \tau _{a}<\tau_{b}\right\} \right)+\lim_{a\downarrow0}\mathrm{c}\left( \left\{ \tau _{b}<t\right\} \right)\notag\\
		 ={}&\lim_{k\uparrow\infty}\frac{\ln^{k}-\ln^{z}}{(k+1)\ln^{k}}+\lim_{k\uparrow\infty}\mathrm{c}\left( \left\{ \tau _{k}<t\right\} \right).
	\end{align}
	Combined with Lemma \ref{taub}, we know that $  \lim_{a\downarrow0}\mathrm{c}\left( \left\{ \tau _{a}<t\right\} \right)=0 $. When $ d>2 $, it can be proved similarly.
	
	(2): When $d=2$, we obtain
	\begin{align*}
		\lim_{t\uparrow\infty}\mathrm{c}\left( \left\{ \tau_{a}>t \right\} \right)\leq{}&\lim_{t\uparrow\infty}\mathrm{c}\left( \left\{\tau_{a}>t, \tau_{a}\leq\tau_{b} \right\} \right)+\lim_{t\uparrow\infty}\mathrm{c}\left( \left\{\tau_{a}>t, \tau_{b}<\tau_{a} \right\} \right)\\
		\leq{}&\lim_{t\uparrow\infty}\mathrm{c}\left( \left\{\tau_{b}>t \right\} \right) +\mathrm{c}\left( \left\{ \tau_{b}<\tau_{a} \right\} \right).
	\end{align*}
    By Lemma \ref{taub} and (\ref{b<a}), 
    \begin{align*}
    	\lim_{t\uparrow\infty}\mathrm{c}\left( \left\{ \tau_{a}>t \right\} \right)\leq{}&\lim_{t\uparrow\infty}\mathrm{c}\left( \left\{\tau_{b}>t \right\} \right) + \frac{\ln^{z}-\ln ^{a}}{\ln ^{b}-\ln ^{a}}\\
    	={}&\frac{\ln^{z}-\ln ^{a}}{\ln ^{b}-\ln ^{a}}\rightarrow0 \quad  \text{ as }b\uparrow\infty.
    \end{align*}
    
	(3): It is clear that $\{\tau_{b}<\tau_{a} \}\uparrow\{\tau_{b}<\tau_{0} \} $ as $ a\downarrow0 $, then we get $ \lim_{a\downarrow0}\mathrm{c}(\{\tau_{b}<\tau_{a}\})=\mathrm{c}(\{\tau_{b}<\tau_{0}\}) $. When $d=1$, by (\ref{b<a}), it leads to 
	\begin{align*}
		\mathrm{c}(\left\{ \tau _{b}<\tau_{0} \right\})={}&\lim_{a\rightarrow0}\mathrm{c}(\left\{ \tau _{b}<\tau_{a} \right\})\\
		={}&(\frac{z}{b})^{\frac{1}{2}}.
	\end{align*}
	Using the same argument in the proof of (2), we obtain that 
	\begin{align*}
		\lim_{t\uparrow\infty}\mathrm{c}\left( \left\{ \tau_{0}>t \right\} \right)\leq{}&\lim_{t\uparrow\infty}\mathrm{c}\left( \left\{ \tau_{b}>t \right\} \right)+\mathrm{c}\left( \left\{ \tau_{b}<\tau_{0} \right\} \right)\\
		={}&\left( \frac{z}{b}\right)^{\frac{1}{2}}\rightarrow0 \quad  \text{ as }b\uparrow\infty.
	\end{align*}
	
	(4): From Remark \ref{remtaub}, we know $ \lim_{b\uparrow\infty}\tau_{b}=\infty $ q.s. Since $ \{\tau_{a}<\tau_{b} \}\uparrow \{\tau_{a}<\infty \}$ as $ b\uparrow\infty $, we get $ \lim_{b\rightarrow\infty}\mathrm{c}(\{\tau_{a}<\tau_{b} \})=\mathrm{c}(\{ \tau_{a}<\infty \}) $. When $ d>2 $, let $ b\rightarrow\infty $ in (\ref{a<b}), it holds that $\mathrm{c}\left( \left\{ \tau _{a}<\infty \right\} \right) =\left( \frac{a}{z}\right)^{\frac{d}{2}-1}$. Define $a_{n}=2^{-n}$ for $ n\in\mathbb{N} $ such that $2^{-n}<z$. Therefore, 
	\begin{equation*}
		\sum_{n=([-\frac{\ln^{z}}{\ln^{2}}]+1)\vee1}^{\infty }\mathrm{c}\left( \left\{ \tau _{a_{n}}<\infty
		\right\} \right) =\sum_{n=([-\frac{\ln^{z}}{\ln^{2}}]+1)\vee1}^{\infty }\left( \frac{2^{-n}}{z}%
		\right) ^{\frac{d}{2}-1}<\infty .
	\end{equation*}%
	By Lemma 5 of \cite{denis2011function}, we obtain 
	\begin{equation*}
		\mathrm{c}\left( \limsup_{n\rightarrow \infty }\left\{ \tau _{a_{n}}<\infty\right\} \right) =0.
	\end{equation*}%
	Then we have $\mathrm{c}\left( \left\{ \inf_{t\geq 0}X_{t}=0\right\} \right)=0$.
\end{proof}

\begin{remark}
	It is worth emphasizing that conclusions (1)–(3) of Theorem \ref{path} are more profound in the capacity sense. In fact, from (1)–(3), we can derive the following conclusions (1')–(3'), 
	\begin{itemize}
		\item[(1')] If $d\geq 2$, $\mathrm{c}\left( \left\{ \tau _{0}<\infty\right\} \right)=0$;
		\item[(2')] If $d=2$, $\mathrm{c}\left( \left\{ \tau_{a}=\infty \right\} \right)=0$;
		\item[(3')] If $d=1$, $\mathrm{c}\left( \left\{ \tau_{0}=\infty \right\} \right)=0$,
	\end{itemize}
    which are closer in form to the classical probabilistic results. However, due to the lack of upper continuity of capacity, the converse implication does not hold. 
\end{remark}

By the above theorem, we can obtain the equation satisfied by the $G$-Bessel
process.

\begin{proposition}
	Let $T\geq0,\ z>0$ and let $d\geq 2$.\ Let $Z$ be a $G$-$\mathrm{BESQ}_{z}^{d}$. Then $\sqrt{Z}$ is a $G$-$\mathrm{BES}_{\sqrt{z}}^{d}$ and satisfies the following equation:
	\begin{equation}\label{bes}
		X_{t}=\sqrt{z}+\frac{d-1}{2}\int_{0}^{t}\frac{1}{X_{s}}I_{\{X_{s}\neq0 \}}\mathrm{d}\left\langle\beta \right\rangle _{s}+\beta _{t},\quad t\in \left[ 0,T\right] .
	\end{equation}%
	Moreover, $\sqrt{Z}\in \tilde{M}_{G}^{2}(0,T)$ is the unique non-negative solution of equation (\ref{bes}) in $ M_{G}^{2}(0,T) $.
\end{proposition}

\begin{proof}
	First, we prove the uniqueness. Assume that $\bar{X}\in M_{G}^{2}(0,T)$ is a non-negative solution of equation (\ref{bes}). Applying the It\^{o} formula to $ \bar{X}_{t}^{2} $, we have 
	\begin{equation*}
		\bar{X}_{t}^{2}=z+2\int_{0}^{t}\bar{X}_{s}\mathrm{d}\beta _{s}+d\left\langle \beta \right\rangle_{t},\quad t\in \left[ 0,T\right] .
	\end{equation*}
	By Proposition \ref{solution besq}, we get $\bar{X}^{2}=Z$ q.s. Since $\bar{X}\geq0$ q.s. we obtain $\bar{X}=\sqrt{Z}$ q.s.
	
	We now turn to show that $\sqrt{Z}$ satisfies equation (\ref{bes}). Since $Z$ is a non-negative process, $\sqrt{Z}$ is well-defined. Notice that $\sqrt{x}\notin C^{2}(\mathbb{R})$, so we can't use the It\^{o} formula directly to $\sqrt{Z}$. Define $\tau_{\epsilon }=\inf \left\{ t\geq 0,Z_{t}\leq \epsilon \right\} $ for $\epsilon >0$. When $t\in [0,\tau _{\epsilon }\wedge T]$, $\sqrt{x}\in C^{2}([\epsilon,\infty ))$, so the It\^{o} formula is still valid. Thus, we have 
	\begin{align*}
		\sqrt{Z_{t}}=\sqrt{z}+\frac{d-1}{2}\int_{0}^{t}\frac{1}{\sqrt{Z_{s}}}\mathrm{d}\left\langle \beta \right\rangle_{s}+\beta _{t},\quad t\in \left[ 0,\tau_{\epsilon}\wedge T\right] .
	\end{align*}
	At this point, we cannot obtain the expression of $\sqrt{Z_{t}},\ t\in(\tau _{\epsilon },T]$ when $\tau _{\epsilon }<T$. From Theorem \ref{path}, we know $\lim_{\epsilon\downarrow0}\mathrm{c}(\{\tau _{\epsilon }< T\})=0 $. Thus, we yield
	\begin{align}\label{zi}
		\hat{\mathbb{E}}\left[ \int_{0}^{T}\left\vert \sqrt{Z_{t}}-\sqrt{Z_{t}}I_{\{\tau_{\epsilon}\geq T\}}\right\vert ^{2}\mathrm{d}t\right] \leq{} &\int_{0}^{T}\hat{\mathbb{E}}\left[ \left\vert \sqrt{Z_{t}}I_{\left\{ \tau _{\epsilon }<T\right\}}\right\vert ^{2}\right] \mathrm{d}t \notag\\
		\leq{} &\int_{0}^{T}\hat{\mathbb{E}}\left[ \left\vert Z_{t}I_{\left\{ \tau_{\epsilon }<T\right\} }\right\vert \right] \mathrm{d}t.
	\end{align}
	By applying the It\^{o} formula to $Z_{t}^{2}$, it leads to
	\begin{align*}
		Z_{t}^{2}=z^{2}+2\int_{0}^{t}Z_{s}^{\frac{3}{2}}\mathrm{d}\beta
		_{s}+\int_{0}^{t}(2d+4)Z_{s}\mathrm{d}\left\langle \beta \right\rangle _{s},\quad t\in \left[ 0,T\right]. 
	\end{align*}%
	Combining with equation (\ref{besq}), we can get 
	\begin{align*}
		\hat{\mathbb{E}}\left[ Z_{t}^{2}\right] \leq{} &z^{2}+\left( 2d+4\right)
		\int_{0}^{t}\hat{\mathbb{E}}\left[ Z_{s}\right] \mathrm{d}\left\langle \beta
		\right\rangle _{s} \\
		\leq{} &z^{2}+\left( 2d+4\right)\bar{\sigma}^{2}\int_{0}^{t}\hat{\mathbb{E}}\left[ Z_{s}\right] \mathrm{d}s \\
		\leq{}&z^{2}+(2d+4)\bar{\sigma}^{2}\int_{0}^{t}(z+d\bar{\sigma}^{2}s) \mathrm{d}s \\
		\leq{} &C_{z,d,\bar{\sigma},T},
	\end{align*}
	where $ C_{z,d,\bar{\sigma},T}=z^{2}+\bar{\sigma}^{2}T(2d+4)(z+d\bar{\sigma}^{2}T) $. According to (\ref{zi}) and the H\"{o}lder inequality, we obtain 
	\begin{align*}
		\hat{\mathbb{E}}\left[ \int_{0}^{T}\left\vert \sqrt{Z_{t}}-\sqrt{Z_{t}}I_{\{\tau_{\epsilon}\geq T\}}\right\vert ^{2}\mathrm{d}t\right] \leq{} &\int_{0}^{T}\hat{\mathbb{E}}\left[ \left\vert Z_{t}I_{\left\{ \tau _{\epsilon }<T\right\}}\right\vert \right] \mathrm{d}t \\
		\leq{} &\int_{0}^{T}\left( \hat{\mathbb{E}}\left[ Z_{t}^{2}\right] \right) ^{\frac{1}{2}}\left( \hat{\mathbb{E}}\left[ I_{\left\{ \tau _{\epsilon}<T\right\} }\right] \right) ^{\frac{1}{2}}\mathrm{d}t \\
		\leq{} &\int_{0}^{T}\left( C_{z,d,\bar{\sigma},T}\right) ^{\frac{1}{2}}\left( \mathrm{c}\left( \left\{ \tau _{\epsilon }<T\right\} \right) \right) ^{\frac{1}{2}}\mathrm{d}t \\
		\leq{} &T\left( C_{z,d,\bar{\sigma},T}\right) ^{\frac{1}{2}}\left( \mathrm{c}\left( \left\{ \tau _{\epsilon }<T\right\} \right) \right) ^{\frac{1}{2}}\rightarrow 0\quad \text{as}\quad \varepsilon \downarrow 0.
	\end{align*}
	Then we get $\sqrt{Z_{t}}=\sqrt{Z_{t}}I_{\{\tau_{\epsilon}\geq T\}}$ q.s. as $\varepsilon \downarrow 0$. Therefore, 
	\begin{equation*}
		\sqrt{Z_{t}}=\sqrt{z}+\frac{d-1}{2}\int_{0}^{t}\frac{1}{\sqrt{Z_{s}}}\mathrm{d}\left\langle\beta \right\rangle _{s}+\beta _{t},\quad t\in \left[ 0,T\right] .
	\end{equation*}
	This completes the proof.
\end{proof}

\section{The Laplace transform of squared $G$-Bessel processes}\label{sec4}

In this section, we present upper and lower bounds on the Laplace transform of the squared $G$-Bessel process.

\begin{proposition}\label{expz}
	Let $T,z\geq0$ and let $d\in \mathbb{N}^{+}$.\ Let $Z$ be a $G$-$\mathrm{BESQ}_{z}^{d}$. Then for each $\lambda >0$ and $ t\in[0,T] $, we have 
	\begin{align}
		\hat{\mathbb{E}}\left[ \exp \left( -\lambda Z_{t}\right) \right]\leq{}&\left( \frac{1+2\lambda \left( \bar{\sigma}^{2}-\underline{\sigma }^{2}\right) t}{1+2\lambda \bar{\sigma}^{2}t}\right) ^{\frac{d}{2}}\exp \left( \frac{-\lambda z}{1+2\lambda \bar{\sigma}^{2}t}\right),\label{expz<}\\ 
		\hat{\mathbb{E}}\left[ \exp \left( -\lambda Z_{t}\right) \right]\geq{}& \left( 1+2\lambda \bar{\sigma}^{2}t\right) ^{-\frac{d}{2}\left( 1+2\lambda\left( \bar{\sigma}^{2}-\underline{\sigma }^{2}\right) t\right) }\exp \left( \frac{-\lambda z\left( 1+2\lambda \left( \bar{\sigma}^{2}-\underline{\sigma }^{2}\right) t\right) }{1+2\lambda \bar{\sigma}^{2}t}\right). \label{expz>}
	\end{align}
\end{proposition}

\begin{proof}
	Define 
	\begin{equation*}
		M_{t}:=\left( 1+2\lambda \left( \bar{\sigma}^{2}T-\left\langle \beta
		\right\rangle _{t}\right) \right) ^{-\frac{d}{2}}\exp \left( \frac{-\lambda
			Z_{t}}{1+2\lambda \left( \bar{\sigma}^{2}T-\left\langle \beta \right\rangle
			_{t}\right) }\right) ,\quad t\in \left[ 0,T\right] .
	\end{equation*}%
	According to the It\^{o} formula, it is easy to show that 
	\begin{equation*}
		\mathrm{d}\left( 1+2\lambda \left( \bar{\sigma}^{2}T-\left\langle \beta \right\rangle_{t}\right) \right) ^{-\frac{d}{2}}=\lambda d\left( 1+2\lambda \left( \bar{\sigma}^{2}T-\left\langle \beta \right\rangle _{t}\right) \right) ^{-\frac{d}{2}-1}\mathrm{d}\left\langle \beta \right\rangle _{t},
	\end{equation*}
    and
	\begin{align*}
		{}&\mathrm{d}\exp \left( \frac{-\lambda Z_{t}}{1+2\lambda \left( \bar{\sigma}
		^{2}T-\left\langle \beta \right\rangle _{t}\right) }\right) \\
	    ={}&\exp \left( \frac{-\lambda Z_{t}}{1+2\lambda \left( \bar{\sigma}^{2}T-\left\langle \beta\right\rangle _{t}\right) }\right)\left(1+2\lambda \left( \bar{\sigma}^{2}T-\left\langle \beta \right\rangle_{t}\right) \right) ^{-1} \left( -2\lambda \sqrt{Z_{t}}\mathrm{d}\beta _{t}-\lambda d\mathrm{d}\left\langle \beta \right\rangle _{t}\right) .
	\end{align*}%
	Therefore,
	\begin{align*}
		\mathrm{d}M_{t} ={}&\left( 1+2\lambda \left( \bar{\sigma}^{2}T-\left\langle \beta\right\rangle _{t}\right) \right) ^{-\frac{d}{2}}\mathrm{d}\exp \left( \frac{-\lambda
		Z_{t}}{1+2\lambda \left( \bar{\sigma}^{2}T-\left\langle \beta \right\rangle_{t}\right) }\right) \\
	    {}&+\exp \left( \frac{-\lambda Z_{t}}{1+2\lambda \left( \bar{\sigma}^{2}T-\left\langle \beta \right\rangle _{t}\right) }\right)\mathrm{d}\left( 1+2\lambda \left( \bar{\sigma}^{2}T-\left\langle \beta \right\rangle
		_{t}\right) \right) ^{-\frac{d}{2}} \\
		={}&\exp \left( \frac{-\lambda Z_{t}}{1+2\lambda \left( \bar{\sigma}^{2}T-\left\langle \beta \right\rangle _{t}\right) }\right) \left( -2\lambda \sqrt{Z_{t}}\left( 1+2\lambda \left( \bar{\sigma}^{2}T-\left\langle \beta\right\rangle _{t}\right) \right) ^{-\frac{d}{2}-1}\right) \mathrm{d}\beta _{t}.
	\end{align*}
	Clearly, $ M $ is a symmetric martingale. Through $ \hat{\mathbb{E}}[M_{T}]=\hat{\mathbb{E}}[M_{0}] $, we have
	\begin{align}\label{mtm0}
		\hat{\mathbb{E}}\left[ \left( 1+2\lambda \left( \bar{\sigma}
		^{2}T-\left\langle \beta \right\rangle _{T}\right) \right) ^{-\frac{d}{2}}\exp \left( \frac{-\lambda Z_{T}}{1+2\lambda \left( \bar{\sigma}^{2}T-\left\langle \beta \right\rangle _{T}\right) }\right) \right]  =\left(1+2\lambda \bar{\sigma}^{2}T\right) ^{-\frac{d}{2}}\exp \left( \frac{-\lambda z}{1+2\lambda \bar{\sigma}^{2}T}\right) .
	\end{align}

	Notice that $\underline{\sigma }^{2}T\leq \left\langle \beta \right\rangle
	_{T}\leq \bar{\sigma}^{2}T$. Then by the monotonicity of the function $x^{-\frac{d}{2}}$, we get
	\begin{equation}\label{fx}
		\left( 1+2\lambda \left( \bar{\sigma}^{2}-\underline{\sigma}^{2}\right) T\right)^{-\frac{d}{2}}\leq \left( 1+2\lambda \left( \bar{\sigma}^{2}T-\left\langle
		\beta \right\rangle _{T}\right) \right) ^{-\frac{d}{2}}\leq 1.
	\end{equation}
	Similarly, we can obtain
	\begin{equation}\label{exp}
		\exp \left( -\lambda Z_{T}\right) \leq \exp \left( \frac{-\lambda Z_{T}}{%
			1+2\lambda \left( \bar{\sigma}^{2}T-\left\langle \beta \right\rangle
			_{T}\right) }\right) \leq \exp \left( \frac{-\lambda Z_{T}}{1+2\lambda
			\left( \bar{\sigma}^{2}-\underline{\sigma}^{2}\right) T}\right).
	\end{equation}
	Combining (\ref{fx}) and (\ref{exp}), we see that 
	\begin{equation*}
		\left( 1+2\lambda \left( \bar{\sigma}^{2}-\underline{\sigma}^{2}\right) T\right)^{-\frac{d}{2}}\exp \left( -\lambda Z_{T}\right) \leq M_{T}\leq \exp \left( \frac{-\lambda Z_{T}}{1+2\lambda \left( \bar{\sigma}^{2}-\underline{\sigma}^{2}\right) T}\right).
	\end{equation*}
    From the monotonicity of $ G $-expectation, there is
	\begin{equation}\label{m<>}
		\hat{\mathbb{E}}\left[ \left( 1+2\lambda \left(\bar{\sigma}^{2}-\underline{\sigma}
		^{2}\right) T\right) ^{-\frac{d}{2}}\exp \left( -\lambda Z_{T}\right) \right]\leq \hat{\mathbb{E}}\left[ M_{T}\right] \leq \hat{\mathbb{E}}\left[ \exp\left( \frac{-\lambda Z_{T}}{1+2\lambda \left( \bar{\sigma}^{2}-\underline{\sigma}^{2}\right) T}\right) \right] .
	\end{equation}

    We now turn to the proof of (\ref{expz<}). According to (\ref{mtm0}) and (\ref{m<>}), we obtain
	\begin{equation*}
		\hat{\mathbb{E}}\left[ \left( 1+2\lambda \left( \bar{\sigma}^{2}-\underline{\sigma}^{2}\right) T\right) ^{-\frac{d}{2}}\exp \left( -\lambda Z_{T}\right) \right]
		\leq \left( 1+2\lambda \bar{\sigma}^{2}T\right) ^{-\frac{d}{2}}\exp \left( \frac{-\lambda z}{1+2\lambda \bar{\sigma}^{2}T}\right) .
	\end{equation*}
	It follows that 
	\begin{equation*}
		\hat{\mathbb{E}}\left[ \exp \left( -\lambda Z_{T}\right) \right] \leq \left( \frac{1+2\lambda \left( \bar{\sigma}^{2}-\underline{\sigma}^{2}\right) T}{1+2\lambda \bar{\sigma}^{2}T}\right) ^{\frac{d}{2}}\exp \left( \frac{-\lambda z}{1+2\lambda \bar{\sigma}^{2}T}\right) .
	\end{equation*}
    For each $ t \in [0,T]$, replacing the terminal time $ T $ with $ t $, then we obtain the desired result according to the above inequality.

    Finally, let us prove (\ref{expz>}). From (\ref{mtm0}) and (\ref{m<>}), we yield
	\begin{equation}\label{<m}
		\left( 1+2\lambda \bar{\sigma}^{2}T\right) ^{-\frac{d}{2}}\exp \left( \frac{-\lambda z}{1+2\lambda \bar{\sigma}^{2}T}\right) \leq \hat{\mathbb{E}}\left[\exp \left( \frac{-\lambda Z_{T}}{1+2\lambda \left( \bar{\sigma}^{2}-\underline{\sigma}^{2}\right) T}\right) \right] .
	\end{equation}
	Due to the H\"{o}lder inequality, it leads to
	\begin{align}\label{mholder}
		\hat{\mathbb{E}}\left[ \exp \left( \frac{-\lambda Z_{T}}{1+2\lambda \left( \bar{\sigma}^{2}-\underline{\sigma}^{2}\right) T}\right) \right] \leq{} &\left( \hat{\mathbb{E}}\left[ \left( \exp \left( \frac{-\lambda Z_{T}}{1+2\lambda\left( \bar{\sigma}^{2}-\underline{\sigma}^{2}\right) T}\right) \right)
		^{1+2\lambda \left( \bar{\sigma}^{2}-\underline{\sigma}^{2}\right) T}\right] \right) ^{\frac{1}{1+2\lambda \left( \bar{\sigma}^{2}-\underline{\sigma}^{2}\right) T}} \notag \\
		={}&\left( \hat{\mathbb{E}}\left[ \exp \left( -\lambda Z_{T}\right) \right]\right) ^{\frac{1}{1+2\lambda \left(\bar{\sigma}^{2}-\underline{\sigma}^{2}\right) T}}.
	\end{align}
	Using (\ref{mholder}) in (\ref{<m}), we have
	\begin{equation*}
		\left( 1+2\lambda \bar{\sigma}^{2}T\right) ^{-\frac{d}{2}}\exp \left( \frac{-\lambda z}{1+2\lambda \bar{\sigma}^{2}T}\right) \leq \left( \hat{\mathbb{E}}
		\left[ \exp \left( -\lambda Z_{T}\right) \right] \right) ^{\frac{1}{1+2\lambda \left( \bar{\sigma}^{2}-\underline{\sigma}^{2}\right) T}}
	\end{equation*}
    Raise both sides of the above inequality to the power of $ 1+2\lambda \left( \bar{\sigma}^{2}-\underline{\sigma}^{2}\right) T $, it follows that
	\begin{equation*}
		\left( 1+2\lambda \bar{\sigma}^{2}T\right) ^{-\frac{d}{2}\left( 1+2\lambda\left( \bar{\sigma}^{2}-\underline{\sigma}^{2}\right) T\right) }\exp \left( \frac{-\lambda z\left( 1+2\lambda \left( \bar{\sigma}^{2}-\underline{\sigma}^{2}\right)T\right) }{1+2\lambda \bar{\sigma}^{2}T}\right) \leq \hat{\mathbb{E}}\left[\exp \left( -\lambda Z_{T}\right) \right] .
	\end{equation*}
	Similarly, we can replace $ T $ with $ t $ in the above inequality and get (\ref{expz>}).
\end{proof}

\begin{remark}
	If $\underline{\sigma}^{2}=\bar{\sigma}^{2}=1$, then $Z$ is reduced to the squared
	Bessel process. In this case, $\hat{\mathbb{E}}\left[ \exp \left( -\lambda
	Z_{t}\right) \right] =\left( 1+2\lambda t\right) ^{-\frac{d}{2}}\exp \left( 
	\frac{-\lambda z}{1+2\lambda t}\right) $, which is consistent with the
	classical case.
\end{remark}

\section{The correspondence between squared $G$-Bessel processes and $ G' $-CIR processes}\label{sec5}

In this section, we will show that there is a class of $ G' $-CIR processes that are space-time transformed squared $ G $-Bessel processes. Before describing the relationship, we need to give a deterministic time transformation formula for the $ G' $-Brownian motion.

\begin{lemma}\label{time-change}
	Let $\beta $ be a one-dimensional $G'$-Brownian motion. Suppose $f$
	is function on $[0,\infty )$ that satisfies $f\in C^{1}([0,\infty ))$ and $f(0)=0$. Furthermore, it is also require that for each $0\leq n<\infty $ and $x\in \lbrack 0,n]$, $f'(x)$ is positive and bounded, i.e., there exists $0<l_{n}\leq L_{n}$ that satisfies $l_{n}\leq f'(x)\leq L_{n}$.
	Then 
	\begin{equation*}
		\int_{0}^{t}\frac{1}{\sqrt{f'(s)}}\mathrm{d}\beta _{f(s)},\quad
		t\geq 0,
	\end{equation*}
	is also a one-dimensional $G'$-Brownian motion.
\end{lemma}

\begin{proof}
	By taking $r=f(s)$, we obtain%
	\begin{equation*}
		\int_{0}^{t}\frac{1}{\sqrt{f'(s)}}\mathrm{d}\beta
		_{f(s)}=\int_{0}^{f(t)}\frac{1}{\sqrt{f'(f^{-1}(r))}}\mathrm{d}%
		\beta _{r}.
	\end{equation*}
	Define 
	\begin{equation*}
		\bar{M}_{t}=\int_{0}^{f(t)}\frac{1}{\sqrt{f'(f^{-1}(r))}}\mathrm{d}
		\beta_{r},\quad\bar{L}_{G}^{p}(\Omega_{t})=\tilde{L}_{G}^{p}(\Omega
		_{f(t)}),\quad \bar{L}_{G}^{p}(\Omega )=\cup _{t=1}^{\infty }\bar{L}_{G}^{p}(\Omega_{t}),\quad\bar{\mathbb{E}}_{t}=\hat{\mathbb{E}}_{f(t)},\quad \bar{\mathbb{E}}=\hat{\mathbb{E}}.
	\end{equation*}
	It can easily be shown that $(\Omega ,\bar{L}_{G}^{1}(\Omega ),\bar{\mathbb{E}})$ is a sublinear expectation space. Since $\frac{1}{\sqrt{f'(f^{-1}(r))}}$ is a bounded and deterministic function, we know that $\bar{M}_{t}\in\bar{L}_{G}^{1}(\Omega _{t})$. By using the BDG inequality (Lemma 1.12 in Chapter \Rmnum{8} of \cite{peng2019nonlinear}), we get 
	\begin{align*}
		\bar{\mathbb{E}}\left[ \left\vert \bar{M}_{t}\right\vert ^{4}\right] ={}&\hat{\mathbb{E}}\left[ \left\vert \int_{0}^{f(t)}\frac{1}{\sqrt{f^{'}(f^{-1}(r))}}\mathrm{d}\beta _{r}\right\vert ^{4}\right]  \\
		\leq{} &c\bar{\sigma}^{4}\hat{\mathbb{E}}\left[ \left\vert \int_{0}^{f(t)}\frac{1}{f'(f^{-1}(r))}\mathrm{d}r\right\vert ^{2}\right]  \\
		={}&c\bar{\sigma}^{4}t^{2}.
	\end{align*}%
	According to Theorem 54 of \cite{denis2011function}, we yield $\bar{M}_{t}\in \bar{L}_{G}^{3}(\Omega _{t})$. For $0\leq t_{1}\leq t_{2}$, it follows that 
	\begin{align*}
		\bar{\mathbb{E}}_{t_{1}}\left[ \bar{M}_{t_{2}}\right]  ={}&\hat{\mathbb{E}}_{f(t_{1})}\left[ \int_{0}^{f(t_{2})}\frac{1}{\sqrt{f'(f^{-1}(r))}}\mathrm{d}\beta _{r}\right]  \\
		={}&\int_{0}^{f(t_{1})}\frac{1}{\sqrt{f'(f^{-1}(r))}}\mathrm{d}\beta
		_{r} \\
		={}&\bar{M}_{t_{1}.}
	\end{align*}%
	From the fact that $\bar{\mathbb{E}}\left[ \bar{M}_{t}\right] =\bar{\mathbb{E}}\left[ -\bar{M}_{t}\right] =0$, $\bar{M}$ is a
	symmetric martingale under $(\Omega ,\bar{L}_{G}^{1}(\Omega),\bar{\mathbb{E}})$. Applying the BDG inequality, we obtain
	\begin{align*}
		\bar{\mathbb{E}}\left[ \left\vert \bar{M}_{t+\epsilon }-\bar{M}
		_{t}\right\vert ^{3}\right]  
		={}&\hat{\mathbb{E}}\left[ \left\vert\int_{f(t)}^{f(t+\epsilon)}\frac{1}{\sqrt{f'(f^{-1}(r))}}\mathrm{d}\beta _{r}\right\vert ^{3}\right]  \\
		\leq{} &c\bar{\sigma}^{3}\hat{\mathbb{E}}\left[ \left\vert
		\int_{f(t)}^{f(t+\epsilon )}\frac{1}{f'(f^{-1}(r))}\mathrm{d}r\right\vert ^{\frac{3}{2}}\right]  \\
		={}&c\bar{\sigma}^{3}\epsilon ^{\frac{3}{2}}.
	\end{align*}
	It is clear that
	\begin{equation*}
		\left\langle \bar{M}\right\rangle _{t}=\int_{0}^{f(t)}\frac{1}{f^{'
			}(f^{-1}(r))}\mathrm{d}\left\langle \beta \right\rangle _{r},\quad t\geq 0.
	\end{equation*}
	Due to Proposition 1.4 in Chapter \Rmnum{4} of \cite{peng2019nonlinear}, we know 
	\begin{equation}\label{mart}
		\hat{\mathbb{E}}_{s}\left[\int_{0}^{t}\frac{1}{f'(f^{-1}(r))}\mathrm{d}\left\langle\beta\right\rangle_{r}-\bar{\sigma}^{2}\int_{0}^{t}\frac{1}{f'(f^{-1}(r))}\mathrm{d}r\right]=\int_{0}^{s}\frac{1}{f'(f^{-1}(r))}\mathrm{d}\left\langle\beta\right\rangle_{r}-\bar{\sigma}^{2}\int_{0}^{s}\frac{1}{f'(f^{-1}(r))}\mathrm{d}r,
	\end{equation}
	where $0\leq s\leq t$. Recall that $\bar{M}$ is a symmetric martingale. Then
	for $0\leq t_{1}\leq t_{2}$, we have 
	\begin{equation}\label{sym mart}
		\bar{\mathbb{E}}_{t_{1}}\left[\bar{M}_{t_{1}}\left(\bar{M}_{t_{2}}-\bar{M}_{t_{1}}\right)\right]=\bar{\mathbb{E}}_{t_{1}}\left[ -\bar{M}_{t_{1}}\left( \bar{M}_{t_{2}}-\bar{M}_{t_{1}}\right) \right] =0.
	\end{equation}
	Combining (\ref{sym mart}), Proposition 3 in \cite{hu2019levy} and (\ref{mart}), it follows that 
	\begin{align*}
		\bar{\mathbb{E}}_{t_{1}}\left[\bar{M}_{t_{2}}^{2}-\bar{\sigma}^{2}t_{2}\right]  ={}&\bar{\mathbb{E}}_{t_{1}}\left[ \left(\bar{M}_{t_{2}}-\bar{M}_{t_{1}}+\bar{M}_{t_{1}}\right) ^{2}\right] -\bar{\sigma}^{2}t_{2} \\
		={}&\bar{\mathbb{E}}_{t_{1}}\left[\left(\bar{M}_{t_{2}}-\bar{M}
		_{t_{1}}\right) ^{2}+2\left( \bar{M}_{t_{2}}-\bar{M}_{t_{1}}\right) \bar{M}_{t_{1}}+\bar{M}_{t_{1}}^{2}\right] -\bar{\sigma}^{2}t_{2} \\
		={}&\bar{\mathbb{E}}_{t_{1}}\left[\left(\bar{M}_{t_{2}}-\bar{M}
		_{t_{1}}\right)^{2}\right]+\bar{M}_{t_{1}}^{2}-\bar{\sigma}^{2}t_{2} \\
		={}&\bar{\mathbb{E}}_{t_{1}}\left[\left\langle\bar{M}\right\rangle_{t_{2}}-\left\langle \bar{M}\right\rangle _{t_{1}}\right] +\bar{M}_{t_{1}}^{2}-\bar{\sigma}^{2}t_{2} \\
		={}&\hat{\mathbb{E}}_{f(t_{1})}\left[\int_{0}^{f(t_{2})}\frac{1}{f'(f^{-1}(r))}\mathrm{d}\left\langle\beta\right\rangle_{r}-\bar{\sigma}
		^{2}t_{2}\right]-\left\langle\bar{M}\right\rangle_{t_{1}}+\bar{M}_{t_{1}}^{2} \\
		={}&\hat{\mathbb{E}}_{f(t_{1})}\left[\int_{0}^{f(t_{2})}\frac{1}{f'(f^{-1}(r))}\mathrm{d}\left\langle\beta\right\rangle_{r}-\bar{\sigma}
		^{2}\int_{0}^{f(t_{2})}\frac{1}{f'(f^{-1}(r))}\mathrm{d}r\right]
		-\left\langle \bar{M}\right\rangle _{t_{1}}+\bar{M}_{t_{1}}^{2} \\
		={}&\int_{0}^{f(t_{1})}\frac{1}{f'(f^{-1}(r))}\mathrm{d}\left\langle
		\beta\right\rangle_{r}-\bar{\sigma}^{2}\int_{0}^{f(t_{1})}\frac{1}{f'(f^{-1}(r))}\mathrm{d}r-\left\langle\bar{M}\right\rangle_{t_{1}}+\bar{M}_{t_{1}}^{2} \\
		={}&\left\langle\bar{M}\right\rangle_{t_{1}}-\bar{\sigma}^{2}t_{1}-\left\langle \bar{M}\right\rangle _{t_{1}}+\bar{M}_{t_{1}}^{2} \\
		={}&\bar{M}_{t_{1}}^{2}-\bar{\sigma}^{2}t_{1}.
	\end{align*}
	Therefore, $\bar{M}_{t}-\bar{\sigma}^{2}t,\ t\geq 0$, is a martingale.
	Similarly, $-\bar{M}_{t}+\underline{\sigma }^{2}t,\ t\geq 0$, is also a
	martingale. From Theorem \ref{Levy}, $\bar{M}$ is a one-dimensional $G'$-Brownian motion.
\end{proof}

\begin{proposition}\label{BESQ-CIR}
	Let $z\geq0$ and let $d\in \mathbb{N}^{+}$. Let $Z$ be the $G$-$\mathrm{BESQ}_{z}^{d}$. Assume $a,b,c>0$ and $\frac{4b}{c^{2}}=d$. Define 
	\begin{equation*}
		X_{t}=e^{-at}Z_{\frac{c^{2}}{4a}\left( e^{at}-1\right) },\quad t\geq0 .
	\end{equation*}
	Then $X$ is a one-dimensional $G'$-CIR process. 
\end{proposition}

\begin{proof}
	Set $f(t)=\frac{c^{2}}{4a}\left( e^{at}-1\right) $. Applying the Ito formula to $X_{t}$, we obtain that 
	\begin{align}\label{dXt}
		\mathrm{d}X_{t} ={}&\mathrm{d}e^{-at}Z_{f(t)}\notag \\
		={}&2e^{-at}\sqrt{Z_{f(t)}}\mathrm{d}\beta_{f(t)}+\frac{4b}{c^{2}}e^{-at}\mathrm{d}\left\langle\beta\right\rangle_{f(t)}-ae^{-at}Z_{f(t)}\mathrm{d}t\notag \\
		={}&-aX_{t}\mathrm{d}t+\frac{4b}{c^{2}}e^{-at}\mathrm{d}\left\langle\beta\right\rangle_{f(t)}+2e^{-at}\sqrt{Z_{f(t)}}\mathrm{d}\beta_{f(t)}.
	\end{align}
	Define 
	\begin{equation*}
		\bar{\beta}_{t}=\int_{0}^{t}\frac{1}{\sqrt{f'(s)}}\mathrm{d}\beta_{f(s)}.
	\end{equation*}
	According to Lemma \ref{time-change}, we konw $\bar{\beta}$ is a one-dimensional $G'$-Brownian motion. Since $ f'>0 $, it follows that   
	\begin{equation*}
		\mathrm{d}\beta _{f(t)}=\sqrt{f'(t)}\mathrm{d}\bar{\beta}_{t},\quad \mathrm{d}\left\langle\beta\right\rangle_{f(t)}=f'(t)\mathrm{d}\left\langle \bar{\beta}\right\rangle _{t}.
	\end{equation*}
	Then (\ref{dXt}) can be transformed into the following form
	\begin{align*}
		\mathrm{d}X_{t} ={}&-aX_{t}\mathrm{d}t+\frac{4b}{c^{2}}e^{-at}\mathrm{d}\left\langle \beta \right\rangle _{f(t)}+2e^{-at}\sqrt{Z_{f(t)}}\mathrm{d}\beta _{f(t)} \\
		={}&-aX_{t}\mathrm{d}t+b\mathrm{d}\left\langle\bar{\beta}\right\rangle_{t}+c\sqrt{X_{t}}\mathrm{d}\bar{\beta}_{t}.
	\end{align*}
    Additionally, $ X_{0}=z $. From equation (5.1) in \cite{A2023GCIR}, $X$ is a one-dimensional $G'$-CIR process.
\end{proof}

According to the above lemma, we can obtain the estimate of the one-dimensional $G'$-CIR process through the squared $G$-Bessel process.

\begin{proposition}
	Let $z\geq 0,d\in\mathbb{N}^{+}$ and let $a,b,c>0$. Let $X,Z$ defined as in Proposition \ref{BESQ-CIR}. Then
	\[  \hat{\mathbb{E}}[X_{t}]=e^{-at}z+\frac{b\bar{\sigma}^{2}}{a}(1-e^{-at}), \ \ \hat{\mathbb{E}}[-X_{t}]=-e^{-at}z-\frac{b\underline{\sigma}^{2}}{a}(1-e^{-at}).\]
\end{proposition}

\begin{proof}
	According to Proposition \ref{BESQ-CIR}, we know 
	\[\hat{\mathbb{E}}[X_{t}]=\hat{\mathbb{E}}[e^{-at}Z_{\frac{c^{2}}{4a}\left( e^{at}-1\right) }].\]
	Set $f(t)=\frac{c^{2}}{4a}\left( e^{at}-1\right) $. Thus, 
	\begin{align*}
		\hat{\mathbb{E}}[e^{-at}Z_{f(t)}]=& \hat{\mathbb{E}}\left[e^{-at}z+2e^{-at}\int_{0}^{f(t)}\sqrt{Z_{s}}\mathrm{d}\beta_{s}+de^{-at}\left\langle \beta \right\rangle _{f(t)}  \right] \\
		=&e^{-at}z+de^{-at}\hat{\mathbb{E}}[\left\langle \beta \right\rangle _{f(t)}].
	\end{align*}
	Due to Proposition 4.1.4 in \cite{peng2019nonlinear}, it follows that $\left\langle \beta \right\rangle _{f(t)}-\bar{\sigma}^{2}f(t), \ t\geq0, $ is a $G$-martingale, which implies
	\[ \hat{\mathbb{E}}[\left\langle \beta \right\rangle _{f(t)}]=\bar{\sigma}^{2}f(t). \]
	Therefore, 
	\begin{align*}
		\hat{\mathbb{E}}[X_{t}]=&e^{-at}z+de^{-at}\hat{\mathbb{E}}[\left\langle \beta \right\rangle _{f(t)}]\\
		=&e^{-at}z+de^{-at}\bar{\sigma}^{2}f(t)\\
		=&e^{-at}z+\frac{b\bar{\sigma}^{2}}{a}(1-e^{-at}).
	\end{align*}
	
	Similarly, we can obtain
	\[\hat{\mathbb{E}}[-X_{t}]=-e^{-at}z-\frac{b\underline{\sigma}^{2}}{a}(1-e^{-at}). \]
	The proof is complete.
\end{proof}

\begin{remark}
	It is worth noting that the conditions of the above proposition are different from those of the three cases in Proposition 5.5 in \cite{A2023GCIR}. Furthermore, the above proposition holds for all $d \geq 0$. However, when $d\notin\mathbb{N}^{+}$, $Z$ is no longer a squared $G$-Bessel process. When $d = b = 0$, the result of the above proposition is consistent with the first case of Proposition 5.5 in \cite{A2023GCIR}.
\end{remark}

\section*{Declarations}

\subsection*{Funding}
This work is supported by the National Natural Science Foundation of China (Grant No. 12326603, 11671231).
\subsection*{Ethical approval}
Not applicable.
\subsection*{Informed consent}
Not applicable.
\subsection*{Author Contributions}
All authors contributed equally to each part of this work. All authors read and approved the final manuscript.
\subsection*{Data Availability Statement}
Not applicable.
\subsection*{Conflict of Interest}
The authors declare that they have no known competing financial interests or personal relationships that could have appeared to influence the work reported in this paper.
\subsection*{Clinical Trial Number}
Not applicable.


\bigskip

\begin{thebibliography}{99}
	\bibitem{A2023path}B. Akhtari, Pathwise convergence under Knight uncertainty, J. Math. Anal. Appl. 518(1) (2023) 126683.                                                                                               
	
	\bibitem{A2023GCIR}B. Akhtari, H. Li, The Cox-Ingersoll-Ross process under volatility uncertainty, J. Math. Anal. Appl. 531(1) (2024) 127867
	
	\bibitem{denis2011function}L. Denis, M. Hu, S. Peng, Function spaces and capacity related to a sublinear expectation: application to $ G $-Brownian motion paths, Potential Anal. 34 (2011) 139–161.
	
	\bibitem{G2003survey}A. G\"{o}ing-Jaeschke, M. Yor, A survey and some generalizations of Bessel processes, Bernoulli 9 (2) (2003) 313–349.
	
	\bibitem{gao2009pathwise}F. Gao, Pathwise properties and homomorphic flows for stochastic differential equations driven by $ G $-Brownian motion, Stochastic Process. Appl. 119 (10) (2009) 3356–3382.
	
	\bibitem{hu2009representation}M. Hu, S. Peng, On representation theorem of $ G $-expectations and paths of $ G $-Brownian motion, Acta Math. Appl. Sin. Engl. Ser. 25 (3) (2009) 539–546.
	
	\bibitem{hu2019levy}M. Hu, X. Ji, G. Liu,  L\'{e}vy's martingale characterization and reflection principle of $ G $-Brownian motion, J. Math. Anal. Appl. 480 (2) (2019) 123436.
	
	\bibitem{hu2021on}M. Hu, X. Ji, G. Liu, On the strong Markov property for stochastic differential equations driven by $ G $-Brownian motion, Stochastic Process. Appl. 131 (2021) 417-453.
	
	\bibitem{hu2016quasi}M. Hu, F. Wang, G. Zheng, Quasi-continuous random variables and processes under the $ G $-expectation framework, Stochastic Process. Appl. 126 (2016) 2367–2387.
	
	\bibitem{hu2025gbes}M. Hu, R. Li, $ G $-Bessel processes and related properties, Acta Math. Sci. 46B(2) (2026) 920–936. 
	
	\bibitem{jea2009math}M. Jeanblanc, M. Yor, M. Chesney, Mathematical Methods for Financial Markets, Springer, Berlin, 2009.
	
	\bibitem{brownian2016}J. Le Gall, Brownian Motion, Martingales, and Stochastic Calculus. Springer-Verlag, Cham, 2016.
			
	\bibitem{liny2013stochastic}Y. Lin, Stochastic differential equations driven by $ G $-Brownian motion with reflecting boundary conditions, Electron. J. Probab. 18 (9) (2013) 23.
	
	\bibitem{liu2020exit}G. Liu, Exit times for semimartingales under nonlinear expectation, Stochastic Process. Appl. 130 (12) (2020) 7338-7362.
	
	\bibitem{peng2004filtration}S. Peng, Filtration consistent nonlinear expectations and evaluations of contingent claims, Acta Math. Appl. Sin. 20 (2004) 1–24.
		
	\bibitem{peng2007G}S. Peng, $ G $-expectation, $ G $-Brownian motion and related stochastic calculus of It\^{o} type, in: Stochastic Analysis and Applications, in: Abel Symp., vol. 2, Springer, Berlin, 2007, pp. 541–567.
	
	\bibitem{peng2008multi}S. Peng, Multi-dimensional $ G $-Brownian motion and related stochastic calculus under $ G $-expectation, Stochastic Process. Appl. 118 (12) (2008) 2223–2253.
	
	\bibitem{peng2019nonlinear}S. Peng, Nonlinear Expectations and Stochastic Calculus under Uncertainty, Springer-Verlag, Heidelberg, 2019.
	
	\bibitem{countinuous2013}D. Revuz, M. Yor, Continuous Martingales and Brownian Motion, Springer-Verlag, Berlin, 1999.
	
	\bibitem{song2011hitting}Y. Song, Properties of hitting times for $ G $-martingales and their applications, Stochastic Process. Appl. 121 (8) (2011) 1770–1784.
		
	\bibitem{song2021grad}Y. Song, Gradient estimates for nonlinear diffusion semigroups by coupling methods, Sci. China Math. 64 (2021) 1093–1108. 
	
	\bibitem{wang2018sample}F. Wang, G. Zheng, Sample path properties of $ G $-Brownian motion, J. Math. Anal. Appl. 467 (1) (2018) 421–431.
	
	\bibitem{xu2009martingale}J. Xu, B. Zhang, Martingale characterization of $ G $-Brownian motion, Stochastic Process. Appl. 119 (2009) 232–248.
	
	\bibitem{xu2010martingale}J. Xu, B. Zhang, Martingale property and capacity under $ G $-framework, Electron. J. Probab. 15 (2010) 2041–2068.
	
	\bibitem{yor2001expon}M. Yor, Exponential Functionals of Brownian Motion and Related Processes, Springe-Verlag, Berlin, 2001.
	
	
\end{thebibliography}
\end{document}